\documentclass[12pt,reqno]{amsart}

\usepackage{etoolbox} 
\patchcmd{\subsubsection}{\normalfont}{\bfseries\textnormal}{}{} 

\usepackage[vmargin=2cm,hmargin=2cm]{geometry} 

\usepackage{amssymb, latexsym, color, etaremune, enumerate, tikz}
\usepackage{caption,float,textcomp,units,xfrac,url,tabularx}

\usetikzlibrary{decorations.markings}

\usepackage[T1]{fontenc}
\usepackage{datetime}   

\DeclareMathOperator{\mmod}{mod}
\global\long\def\epsilon{\varepsilon}

\newtheorem{theorem}{Theorem}[section]

\newtheorem{lemma}[theorem]{Lemma}
\newtheorem{claim}[theorem]{Claim}
\newtheorem{corollary}[theorem]{Corollary}

\theoremstyle{definition}
\newtheorem{definition}[theorem]{Definition}

\newtheorem{notation}[theorem]{Notation}
\newtheorem*{conjecture}{Conjecture}

\newcommand{\tref}[1]{Theorem \ref{#1}}
\newcommand{\dref}[1]{Definition \ref{#1}}
\newcommand{\lref}[1]{Lemma \ref{#1}}
\newcommand{\cref}[1]{Corollary \ref{#1}}

\newcommand{\clref}[1]{Claim \ref{#1}}

\begin{document} 

\title[Proof of Sister Beiter cyclotomic coefficient conjecture]{A proof of the corrected Sister Beiter cyclotomic coefficient 
conjecture inspired by Zhao and Zhang}

\author{B. Juran}
\address[B. Juran]{Department of Mathematical Sciences, Universitetsparken 5, 2100 K{\o}benhavn {\O}, Denmark}
\email{bj@math.ku.dk}

\author[P. Moree]{P. Moree}
\address[P. Moree]{Max-Planck-Institut f\"ur Mathematik, Vivatsgasse 7, D-53111 Bonn, Germany}
\email{moree@mpim-bonn.mpg.de}

\author{A. Riekert}
\address[A. Riekert]{Applied Mathematics: Institute for Analysis and Numerics, University of M\"unster, Orleans-Ring 10, D-48149 M\"unster, Germany}
\email{ariekert@uni-muenster.de}

\author{D. Schmitz}
\address[D. Schmitz]{Fachbereich Mathematik, Technische Universit\"at Darmstadt, Schlossgartenstrasse 7, D-64289 Darmstadt, Germany}
\email{david.schmitz@stud.tu-darmstadt.de}

\author{J. V\"ollmecke}
\address[J. V\"ollmecke]{
Rheinische Friedrich-Wilhelms-Universit\"at Bonn, Regina-Pacis-Weg 3, D-53113 Bonn, Germany}
\email{julian.voellmecke@uni-bonn.de}


\maketitle

\begin{abstract}
\noindent The largest 
coefficient (in absolute value) 
of a cyclotomic polynomial $\Phi_n$ is called its height $A(n)$.
In case $p$ is a fixed prime it turns out that as $q$ and $r$ range over all primes satisfying
$p<q<r$, the height $A(pqr)$ assumes a maximum $M(p)$. 
In 1968, Sister Marion Beiter conjectured that 
$M(p)\leq (p+1)/2$. In 2009, this
was disproved for every $p\ge 11$ by Yves Gallot and Pieter Moree. 
They proposed a Corrected Beiter Conjecture, namely $M(p)\leq 
2p/3$. In 2009, Jia Zhao and Xianke Zhang posted on the arXiv what they thought to be a proof of this conjecture. Their work was never accepted for 
publication in a journal. However, in retrospect it turns out to be essentially correct, but rather sketchy at some points. Here we supply 
a lot more details.  
\par The bound 
$M(p)\le 2p/3$ allows us to 
improve some
bounds of Bzd\k{e}ga from 2010 for ternary cyclotomic coefficients.
It also makes it possible to determine
$M(p)$ exactly for three new primes $p$ and study the fine structure of $A(pqr)$ for
them in greater detail.
\end{abstract}
\section{Introduction}
The $n^{th}$ cyclotomic polynomial $\Phi_n(x)$ is defined by

\begin{equation*}
\Phi_n(x)=\prod_{\substack{1 \leq j \leq n\\ (j,n)=1}} \left(x-\mathrm e^{\frac{2 \pi \mathrm i j }{n}}\right)=\sum_{k=0}^{\varphi(n)} a_n(k)x^k,
\end{equation*}
where $\varphi$ denotes Euler's totient function.
This definition implies that
\begin{equation}
\label{factor}
x^n-1=\prod_{d \mid n} \Phi_d(x).
\end{equation}
The cyclotomic polynomials have integer 
coefficients and are irreducible over the rationals. Hence the product 
\eqref{factor} gives a factorization of $x^n-1$ into irreducible polynomials. We define the \textsl{height} $A(n)$ of $\Phi_n$ as
\begin{equation*}
A(n)=\max\{\left| a_n(k) \right|: 0\leq k \leq \varphi(n)\}.
\end{equation*}
As $\Phi_{p^2 n}(x) = \Phi_{pn}(x^p)$ for every prime $p$ and therefore $A(p^2 n) = A(pn)$ as well as $\Phi_{2n}(x) = \Phi_n(-x)$ for $2\nmid n$,  it follows that $A(n)=A(m)$ where $m$ is the largest odd squarefree factor of $n$. Further, $A(n)=1$ whenever $n$ has at most two distinct odd prime factors (see \lref{a_pq}) and so the easiest case where we can expect non-trivial behaviour of $A(n)$ is the so-called \emph{ternary} case $n=pqr$, where $3\le p<q<r$ are primes. The smallest ternary integer is 105 and one has $A(105)=2$. In 1895, Bang \cite{Bang} proved that $A(pqr) \leq p-1$. This implies the existence of
\begin{equation*}
M(p):= \max_{p<q<r\text{ primes}} A(pqr)
\end{equation*}
Noticing that $A(105)=2$, Bang showed that $M(3)=2$.
\par The biggest open problem in the theory of ternary cyclotomic polynomials is to find a formula or efficient algorithm for computing $M(p)$ for every prime $p$. Since the  set $\{(q,r): p<q<r\}$ is infinite, this is non-trivial. Duda \cite{Duda} gave an algorithm for determining $M(p)$. 
Unfortunately it requires the evaluation of coefficients of $\Phi_n$ with $n=O(p^{21}),$ making its running time so bad that we could not use it to even find one new value of $M(p)$.
\par An easier problem is to give an upper bound for $M(p)$. 
In 1968, Sister Marion Beiter \cite{Bei0}
conjectured that $A(pqr)\le (p+1)/2$ and proved her conjecture in
case $q\equiv \pm 1\,({\rm mod~}p)$ or 
$r\equiv \pm 1\,({\rm mod~}p)$ (a result
obtained independently by Bloom \cite{Bloom}). As a corollary one
obtains Bang's result that $M(3)=2$. 
In 1978, Sister Beiter \cite{Bei78} went beyond this and investigated when
$A(3qr)=2.$
In 1971, 
Sister Beiter \cite{Bei} showed that $M(p)\le p-\lfloor \frac{p+1}4\rfloor$.
Bloom \cite{Bloom} showed that $M(5)=3$.
M\"oller \cite{Moe}, building on work by Emma Lehmer \cite{Emma}, proved that $M(p) \geq (p+1)/2$ for all $p$, so Beiter's conjecture would imply that $M(p)={(p+1)}/{2}$. However, Moree and Gallot \cite{GM} disproved the Beiter conjecture for all primes $p \geq 11$. Further they proposed the following weaker conjecture.
\begin{conjecture}[Corrected Beiter Conjecture, 2009] 
We have $M(p) \leq \frac{2}{3}p.$
\end{conjecture}
In 2019, Luca et al.~\cite{ternary}, using techniques from analytic number theory, made some partial progress.
\begin{theorem}[Luca et al.~\cite{ternary}]
\label{cor2}
The relative density of ternary integers $pqr$ for which $A(pqr)\le 2p/3$ is at least $0.925$.
\end{theorem}
\noindent The Corrected Beiter Conjecture implies of course that $A(pqr)\le 2p/3$ for all ternary integers.
\par The best general upper bound for ternary cyclotomic coefficients is due to Bzd\k{e}ga \cite{Bzdega} (improving on an earlier upper bound due to Bachman \cite{B1}), and will be given in Section 2 (\tref{Bzd}). Using this bound, Zhao and Zhang \cite{ZZ1} proved that $M(7)=4$, thus establishing both the Corrected Beiter Conjecture and Beiter's original conjecture for $p=7$. Up to now, no values of $M(p)$ for $p>7$ are explicitly known. 
\par In 2009, Lawrence \cite{Brian} announced a proof of the Correct Beiter Conjecture for all primes $p>10^6$, but details were 
never published. In the same year, Zhao and Zhang \cite{ZZ2} posted a purported proof of the Corrected Beiter Conjecture on the arXiv. It builds upon and extends the methods they used to show that $M(7)=4$ in 
\cite{ZZ1}, but is rather longer and more involved. Unfortunately, their paper was never published in a 
journal and the conjecture is still regarded as being open. However, on carefully checking the alleged proof of Zhao and Zhang, we noted that the main ideas were correct. 
In this paper we will establish the Corrected Beiter Conjecture. 
\begin{theorem}
\label{cbc}
We have $$M(p) \leq \frac{2}{3}p.$$
\end{theorem}
\begin{corollary}
We have $$\lim_{p\rightarrow\infty}\frac{M(p)}{p}=\frac{2}{3}.$$
\end{corollary}
\begin{proof}
This follows on combining \tref{cbc} and 
the result of Gallot and Moree \cite{GM} that given any $\epsilon > 0$, the inequality 
$M(p) > \left( {2}/{3} - \epsilon \right)p$ holds for all large enough primes $p$. 
\end{proof}
\begin{corollary}
We have $M(11)=7$, $M(13)=8$ and $M(19)=12$. 
\end{corollary}
\begin{proof}
That $7,8,$ respectively $12$ are upper bounds 
follows by the theorem, that they are lower bounds
follows from the examples given in Table 1.
\end{proof}
\vfil\eject

\centerline{\bf TABLE 1}
\begin{center}
\begin{tabular}{|c||c|c|c|c|c|}
\hline
$p$ & $M(p)$ & smallest $n$ & smallest $k$ & $a_{pqr}(k)$ & source\\
\hline
\hline
3 & 2 & $3\cdot 5\cdot 7$ & 7 & -2 & Bang \cite{Bang}\\
\hline
5 & 3 & $5\cdot 7\cdot 11$ & 119 & -3 &  Bloom \cite{Bloom}\\
\hline
7 & 4 & $7\cdot 17\cdot 23$ & 875 & 4 & Zhao and Zhang \cite{ZZ1}\\
\hline
11 & 7 & $11\cdot 19\cdot 601$ & 34884 & 7 & this paper\\
\hline
13 & 8 & $13\cdot 73\cdot 307$ & 89647 & 8 & this paper\\
\hline
19 & 12 & $19\cdot 53\cdot 859$ & 318742 & -12 & this paper\\
\hline
\end{tabular}
\end{center}
For every prime $p$ in the table, the smallest
$n=pqr$ is given with $A(n)=M(p)$ (as computed by Yves Gallot). For this value of $n$ the smallest $k$ is given such that $|a_n(k)|=M(p)$ (as computed by Bin Zhang). 
\par The connaisseur of the
cyclotomic polynomial literature might find Table 1
look familiar. Indeed, it already appears in Gallot et al.\,\cite{GMW} (but without the ``smallest $k$'' column).
The latter paper, written shortly after the appearance
of the preprint by Zhao and Zhang \cite{ZZ2}, had optimistically assumed
their work to be correct. Various results in \cite{GMW} are
thus proved taking the Corrected Beiter Conjecture for 
granted. There are
too many of these results to be listed here. In each case it
easily follows from the proofs given whether the conjecture was assumed or not. In any 
case, all of the results in Gallot et al.\,\cite{GMW} can now be
regarded as proved.

We conjecture that there is a sharpening of $M(p)\le 2p/3$ possible in the sense that there exists a function $f(p)$ tending to infinity with $p$ such that $M(p)\le 2p/3-f(p)$ for every prime $p$. 
The growth of the function $f$ must be rather modest
as Cobeli et al.~\cite{CGMZ} showed that 
$M(p)>2p/3-3p^{3/4}\log p$ for all primes $p$ and
$M(p)>2p/3-c\sqrt{p}$ for infinitely many primes 
$p$ for some $c>0$.
In particular we make the following conjecture.
\begin{conjecture}
Given any real number 
$r$ there exist only finitely many
primes $p$ such that $M(p)> 2p/3-r$.
\end{conjecture}
For example, if $r=1$ we believe that 
the primes $p$ satisfying $M(p)>2p/3-1$ 
are precisely those listed 
in Table 1. 
\par It seems that our conjecture
cannot be proved using our method of 
proof of Theorem \ref{cbc}.
\subsection{On the frequency of $A(pqr)=M(p)$}
\label{sec1.1}
It is natural question how often the maximum $M(p)$ is reached. To 
study this
it is helpful to consider the quantity 
\begin{equation}
\label{Mpq}
M(p;q)=\max\{A(pqr): 2<p<q<r\},
\end{equation}
where $p$ and 
$q$ are fixed and $r$ ranges over the primes $>q.$ 
Gallot et al.\,\cite{GMW} were the first to introduce and study 
$M(p;q)$.
They described a rather efficient finite procedure to compute it.
Duda \cite{Duda}, using 
a rather geometric method, showed that if $q>14p^{10},$ then the value of $M(p;q)$ only depends on the
congruence class $\beta$ of $q$ modulo $p.$
This value we denote by $M_{\beta}(p).$ Thus
\begin{equation}
\label{dominik}
M_{\beta}(p)=M(p;q),{\rm~for~any~}q>14p^{10}{\rm~satisfying~}
q\equiv \beta\,({\rm mod~}p).
\end{equation}
Duda showed further that
$$M_{\beta}(p)=\max\{M(p;q):q>p,~q\equiv \beta\,({\rm mod~}p)\},$$
and established the symmetry $M_{\beta}(p)=M_{p-\beta}(p).$ It follows
that
\begin{equation}
\label{Mbetaset}    
M(p)=\max\{M_1(p),\ldots,M_{\frac{p-1}2}(p)\}.
\end{equation}
This formula in combination 
with \eqref{dominik} yields a finite procedure to determine $M(p),$ but unfortunately it
is very inefficient. The bound $14p^{10}$ is presumably far from optimal, 
but it is certainly at least $\gg p^2.$
This follows from the result of Gallot et al.\,\cite{GMW} that
\begin{equation}
\label{1modp}
M(p;q)=\min\Big\{\frac {q-1}p+1,{\frac{p+1}{2}}\Big\}{\rm~if~}q\equiv 1\,({\rm mod~}p).
\end{equation}
We infer that $M_1(p)=M_{p-1}(p)=\frac {p+1}2$ (showing the 
correctness of the $\beta=1$ column in Table 2).
If $q/p$ is small, then frequently $M(p;q)<M_{\beta}(p)$ and 
\eqref{1modp} gives an example of this. As 
mentioned earlier, Sister Beiter \cite{Bei0} and, independently 
Bloom \cite{Bloom}, already proved that $M(p;q)\le \frac{p+1}2$ if 
$q\equiv 1\,({\rm mod~}p)$.
\par It follows from \eqref{dominik} and Dirichlet's theorem on primes in arithmetic
progression that the set of primes $q\equiv \beta\,({\rm mod~}p)$ for
which $M(p;q)=M_{\beta}(p)$ has natural density $\frac 1{p-1}$.
Since  $M_{\beta}(p)=M_{p-\beta}(p)$  we infer that the set of primes $q$ for which 
$M(p;q)=M(p)$ has a natural density $\delta(p)$ satisfying $\delta(p)\ge\frac 2 {p-1}.$
\begin{conjecture}
Let $2<p\le 19$ be a prime $\ne 17$. Then $M_{\beta}(p)$ is given in
Table 2 (keeping in mind that $M_{p-\beta}(p)=M_{\beta}(p)$).\\

\centerline{{\bf TABLE 2}: Values of $M_{\beta}(p)$}
\begin{center}
\begin{tabular}{|c||c||c|c|c|c|c|c|c|c|c|}
\hline
$p\backslash \beta$ & $M(p)$ & $1$ & $2$ & $3$ & $4$ & $5$ & $6$ & $7$ & $8$ & $9$\\
\hline
\hline
$3$ &  $2$ &{\bf 2} &  &  &  &  & & & &\\
\hline
$5$ &  $3$ &{\bf 3} & {\bf 3} &  &  &  & & & &\\
\hline
$7$ &  $4$ &{\bf 4} & {\bf 4} & {\bf 4} &  &  & & & &\\
\hline
$11$ & $7$  &{\bf 6} & {\bf 6} & {\bf 7} & {\bf 7} & $6$ & & & &\\
\hline
$13$ &  $8$ & {\bf 7} & {\bf 7} & $7$ & {\bf 8} & {\bf 8} & $7$& & &\\
\hline
$19$ &  $12$ & {\bf 10} & {\bf 10}  & $10$  & {\bf 12}  & $11$ & $9$ & {\bf 11} & $11$ & $10$\\
\hline
\end{tabular}
\end{center}
In particular, the set of primes $q$ such
that $M(p;q)=M(p)$ has density $\delta(p)$ as given in the table 
below.
\begin{center}
\begin{tabular}{|c|c|c|c|c|c|c|c|}
\hline
$p$ & $3$ & $5$ & $7$ & $11$ & $13$ & $19$ \\
\hline
$\delta(p)$ & $1$ & $1$ & $1$ & $2/5$ & $1/3$ & $1/9$\\
\hline
\end{tabular}
\end{center}
\end{conjecture}
It follows using the next theorem 
that these conjectural densities are certainly \emph{lower bounds} for $\delta(p)$, a result
that improves Theorem 5 of \cite{GMW} in case $p=13$.
\par We define the auxiliary function
\begin{equation}
    \label{auxiliary}
m(j)=\min\Big\{w(j),\Big\lfloor \frac{2}3 p \Big\rfloor\Big\},\text{~with~}w(j)=
\begin{cases}
\frac {p-1}2+j & {\rm ~if~}j<\frac p4;\\
p-j & {\rm ~if~}\frac p4<j\le\frac {p-1}2;\\
w(p-j) & {\rm~if~}j>\frac {p-1}2.
\end{cases}
\end{equation}
\begin{theorem}
\label{tabletheorem}
Let $v$ be an entry in the $p$-row and $\beta$-column with 
$\beta\le \frac{p-1}2$ in Table {\rm 2}. Let $\beta^*$ denote the inverse 
of $\beta$ in the interval $[1,p-1]$. 
We have $$v\le M_{\beta}(p)\le m(\beta^*),$$
and, moreover, if $v$ is in boldface, then $M_{\beta}(p)=v$, 
\end{theorem}
The final entries of the rows in Table 2 satisfy
$\frac {p+1}2\le M_\frac {p-1}2(p)\le \min\{\frac {p+3}2,\big\lfloor \frac{2}3 p\big\rfloor\}$, for example.
\par Without the bound $M(p)\le\frac 23p$ at hand only a substantial weaker 
version of Theorem \ref{tabletheorem} can be proven. Thus Theorem \ref{tabletheorem} can be seen
as an application of our main result.
\par As a further application we generalize in Section \ref{bartekbounds} some bounds
of Bzd\k{e}ga (\tref{Bzd}) and show that $M_{\beta}(p)\le m(\beta^*)$ for every
odd prime $p$ and $1\le \beta\le p-1$ (\tref{BBimproved}).
\par There are of course more aspects
to cyclotomic coefficients than discussed in this paper, for a recent survey see 
Sanna \cite{Sanna}.
\section{Preliminaries}
\begin{definition}
Let 
\begin{equation*}
\Phi_n(x)=\sum_{k=0}^{\varphi(n)} a_n(k)x^k,
\end{equation*}
with
$a_n(k)$ the $k^{\text{th}}$ coefficient of the $n^{\text{th}}$ cyclotomic polynomial. For $k < 0$ or $k > \varphi(n)$, we define $a_n(k) = 0$.
\end{definition}
\begin{definition} \label{ov}
For any two distinct primes $p$ and $q$ let $0<\overline{p_q} \leq q-1$ be the unique integer with $\overline{p_q} \equiv p\,(\mmod q)$.
\end{definition}
\begin{definition} \label{*} 
For any two distinct primes $p$ and $q$ let $0<p_q^*\leq q-1$ be the inverse of $p\,(\mmod{q})$, i.e. the unique integer in $[1, q-1]$ with $pp_q^* \equiv 1\,(\mmod q)$. Likewise modulo $p$ we define $q_p^*$ and $r_p^*$.
\end{definition}
\begin{lemma} \label{pqinv}
	We have $pq+1 = pp_q^*+qq_p^*$ for any two distinct primes $p$ and 
	$q$.
\end{lemma}
\begin{proof}
Note that $1 < qq_p^*+pp_q^* < 2pq $ and that both modulo $p$ and modulo $q$ the integer $qq_p^*+pp_q^*$ equals $1$ 
and hence, by the Chinese Remainder Theorem, also $1$ modulo $pq$. 
\end{proof}

\par The following lemma is well-known, 
Lam and Leung \cite{LL} is an easily accessible
reference. It can also be interpreted in terms of numerical
semigroups, see, e.g., Moree \cite{AMS}.
\begin{lemma}\label{a_pq} Let $p<q$ be primes.
The $k^{\text{th}}$ coefficient $a_{pq}(k)$ of $\Phi_{pq}$ satisfies
\begin{equation*} 
a_{pq}(k)=\begin{cases} 1&\text{if $k=up+vq$ for some $u \in [0,\; p_q^*-1]$ and $v \in [0,\; q_p^*-1]$}; \\ -1&\text{if $k=up+vq-pq$ for some $u\in [p_q^*,\; q-1]$ and $v\in [q_p^*,\; p-1]$}; \\ 0&\text{otherwise.} \end{cases}
\end{equation*}
\end{lemma}
Since $u\equiv kp_q^*\,(\mmod q)$ and $v\equiv kq_p^*\,(\mmod p)$, $k$ uniquely determines $u$ and $v$, so we can denote $u=[k]_p,v=[k]_q$ for every $0\leq k\leq pq$. Observe that 
$$[k]_p=[k']_p\iff k\equiv k'\,(\mmod q) \text{~and~}[k]_q=[k']_q\iff k\equiv k'\,(\mmod p).$$ 
A further result we will make use of is due to Nathan Kaplan.
\begin{lemma}[Kaplan \cite{Kap}, 2007]  \label{Kap}
If $p<q<r$, we have
\begin{equation*}
A(pqr)=A(pqs)
\end{equation*}
for every $s>q$ with $s \equiv \pm r\,(\mmod{pq})$.
\end{lemma}
\begin{corollary}
\label{rp-rswap}
If $p<q<r_1$ are primes, there exists a prime $r>q$ such that $A(pqr_1)=A(pqr)$ and $r_p^*=p-(r_1)_p^*$.
\end{corollary}
\begin{proof}
By Dirichlet's theorem we can choose $k$ such that $r=-r_1+kpq > q$ is a prime 
and hence $A(pqr_1)=A(pqr)$ by \lref{Kap}. As $r\equiv -r_1\,(\mmod p)$, we have $r_p^* = p-(r_1)_p^*$.
\end{proof}
This lemma only allows to exchange the largest prime. There is also an
generalization that does not have this restriction.
\begin{lemma}[Bachman and Moree \cite{BM}, 2011]
\label{BM-swap} 
If $r\equiv \pm s\bmod{pq}$  and $r > \max\{p,q\} > s\ge 1$, then
$A(pqs)\le A(pqr)\le A(pqs)+1$.
\end{lemma}
The upper bound part we do not need, but we like to point out that it can happen that 
$A(pqr)=A(pqs)+1$.
\begin{corollary}
\label{qrswap}
Let $p<q_1<r_1$ be primes. Then there exist primes $q<r$ such that $A(pqr)\geq A(pq_1r_1)$,
$q_p^*=(r_1)_p^*$ and $r_p^*=(q_1)_p^*$.
\end{corollary}
\begin{proof}
By Dirichlet's theorem, we can choose $k\ge 1$ such that
$r=q_1+kpr_1$ is a prime. Note that
$r>\max\{p,r_1\}>q_1\ge 1$. Applying \lref{BM-swap}, we find $A(pq_1r_1)\leq A(prr_1)
=A(pr_1r)$. Thus, we can simply take $q=r_1$.
\end{proof}
This corollary immediately implies the next one.
\begin{corollary}
\label{qrswapineq}
Let $p<q_1<r_1$ be primes. Then there exist primes $q<r$ such that $A(pqr)\geq A(pq_1r_1)$ and $q_p^* \leq r_p^*$.
\end{corollary}
Applying Corollary 
\ref{qrswap}, followed by Corollary \ref{rp-rswap} and then Corollary \ref{qrswap} again yields the
next corollary.
\begin{corollary}
\label{qp-qswap}
Let $p<q_1<r_1$ be primes. Then there exist primes $q,r$ such that $p<q<r$, $A(pqr)\geq A(pq_1r_1)$ and $r_p^*=(r_1)_p^*$, $q_p^*=p-(q_1)_p^*$.
\end{corollary}
Using three of these four corollaries  we can prove
the following lemma which is crucial for
our proof.
\begin{lemma}
\label{ineq}
\label{chain}
Let $p<q_1<r_1$ be primes. Then there exist primes $q,r$ such that $A(pqr)\geq A(pq_1r_1)$ and 
$1 \leq p-q_p^* \leq r_p^*<p-r_p^* \leq q_p^* \leq p-1$.
\end{lemma}
\begin{proof}
It is easy to check that
the following algorithm will produce the
required chain of inequalities.\\
{\tt Algorithm}\\
Check if $\min\{q_p^*,p-q_p^*\}\le 
 \min\{r_p^*,p-r_p^*\}$.\\ 
\indent  If NO, swap $r_p^*$ with
 $q_p^*$ using Corollary \ref{qrswap}.\\ 
 Check if $q_p^*>(p-1)/2$.\\
 \indent  If NO, swap $q_p^*$ with $p-q_p^*$ using Corollary \ref{qp-qswap}.\\
Check if $r_p*\le (p-1)/2$.\\
 \indent If NO, swap $r_p^*$ and 
 $p-r_p^*$ using Corollary \ref{rp-rswap}.
\end{proof}

The following bound plays an important role in the proof of
\tref{cor2} given in Luca et al.\,\cite{ternary} and the proof of 
\tref{cbc} given in this paper. 
\begin{theorem}[Bzd\k{e}ga \cite{Bzdega}, 2009] \label{Bzd} 
Let $p<q<r$ be primes. Let
\begin{equation*}
\alpha = \min\{q_p^*,~p-q_p^*,~r_p^*,~p-r_p^*\}
\end{equation*}
and let $0<\beta\leq p-1$ the unique integer with $\alpha\beta qr \equiv 
1\,(\mmod{p})$. 
We have
\begin{equation*}
a_{pqr}(i) \leq \min\{2\alpha + \beta,~p-\beta\}\text{~and~}-a_{pqr}(i) \leq \min\{p+2\alpha - \beta,~\beta\}.
\end{equation*}
\end{theorem}

\par We define a function $\chi$ that will play a major role in our
proof.
\begin{definition} \label{chi}
Let $k$ and $m$ be integers, $p,q$ and $r$  primes satisfying $q>p$. We 
put
\begin{equation*} 
\chi_k(m)=\begin{cases} 1&\text{if there exists $s\in \mathbb Z$ with 
$mr+q< k+1+spq \le mr+q+p$;} \\ -1&\text{if there exists $s\in \mathbb Z$ with 
$mr< k+1+spq \le mr+p$;} \\ 0&\text{otherwise.} \end{cases} 
\end{equation*}
 {}From the definition it is immediate that $\chi_k(m)$ only depends on the values of $k$ and $m$ modulo $pq$.
Note that in order for $\chi_k(m)$ to be non-zero, the integer $k+1+spq$ has to be an element of one of
two disjoint strings of $p$ consecutive integers.
In order to show that Definition \ref{chi} is consistent, we have to show
that we cannot both have
$mr+q < k+1+s_1pq \le mr+q+p$ for some $s_1$ and 
 $mr < k+1+s_2pq \le mr+p$ for some $s_2$. Indeed, if both inequalities
 were to hold simultaneously, then it would follow that
$0<q-p<(s_1-s_2)pq<p+q<2q$, which is impossible.
 \par The following theorem expresses the coefficients of $\Phi_{pqr}$ using the coefficients of $\Phi_{pq}$, which can be calculated using \lref{a_pq}, and our function $\chi$ from \dref{chi}. 
\end{definition}
\begin{theorem}[Zhao and Zhang \cite{ZZ1}, 2010] \label{a_pqr} 
Let $p<q<r$ be primes. 
Let $m_0$ be the smallest integer such that 
$m_0r+p+q \geq k+1+pq$.
Then
\begin{equation*}
a_{pqr}(k)=\sum_{m\geq m_0} a_{pq}(m)\chi_k(m).
\end{equation*}
\end{theorem} 
\noindent The proof
makes use of the identity $$\Phi_{pqr}(x)(x^{pq}-1)(x-1)=\Phi_{pq}(x^r)(x^p-1)(x^q-1),$$ which allows
one to relate a ternary coefficient to a sum of binary ones. 
Note that $m_0$ in Theorem 
\ref{a_pqr} equals $\left\lceil\frac{k+(p-1)(q-1)} r\right\rceil$.
\begin{corollary} We have
\begin{equation*}
A(pqr) \leq \max_{j,k \in \mathbb{Z}} \left| \sum_{m \geq j} a_{pq}(m)\chi_k(m) \right|.
\end{equation*}
\end{corollary}
\noindent It turns out that the latter inequality is actually an
equality \cite[Lemma 2.3]{ZZ1}.
\par A further result of Zhao and Zhang we need is the following.
\begin{theorem}[Zhao and Zhang \cite{ZZ1}, 2010] \label{sum0} We have
\begin{equation*}
\sum_{m\ge 0} a_{pq}(m)\chi_k(m)=0.
\end{equation*}
\end{theorem}
\noindent A short proof of this theorem can be given using 
Theorem \ref{a_pqr} and the observation that the value of $\chi_k(m)$ only
depends on the residue class of both $k$ and $m$ modulo $pq$.

\section{Proof of the Corrected Beiter Conjecture}
Our proof will mainly use the methods from the unpublished article \cite{ZZ2} by Zhao and Zhang. Many of the ideas 
there are already 
to be found in
their published paper \cite{ZZ1} (dealing with
the case $p=7$). 
Their proof that $M(7)=4$ 
can be regarded as a baby version of the proof we are going to present.
\par We start by giving a glossary of the notation introduced in the course of our proof.
\vskip .3cm
\centerline{\bf Glossary}
\begin{center}
\begin{tabular}{r c l}
$[\cdot]$ & & Definition \ref{u}\\
$C_{1,1},C_{1,2},C_{2,1},C_{2,2}$ & & Defined just below Lemma \ref{shift}\\
low integer & & Integer in $[0,q_p^*-1]$\\
high integer &  & Integer in  $[q_p^*,p-1]$\\ 
$I_p$ &  &\{0\,,\ldots,\,p-1\}\\
$S$ & & Special integers (\dref{largedef})\\
$P,P^+,P^-$ & & Plain integers (ibid.)\\
$N$ & & Null integers: $N=I_p \setminus (P \cup S)$  \\
$v_0$ & & Largest low special integer\\
$h(v),h_q(v)$ & & \dref{hqdef}\\
map $f$ & & \dref{f}\\
$S_0$ & & $\{v \in S:f(v) \in P\}$\\
map $g$ & & \dref{g}\\
\end{tabular}
\end{center}
\vskip .3cm
In some clearly identifiable cases the interval notation $[a,b]$ is used to denote
the set $\{a,a+1,\ldots,b\}.$
\par Since $M(2)=1$, $M(3)=2$, $M(5)=3$ 
and $M(7)=4$, the conjecture is true for $p\le 7$. 
We will argue by contradiction and assume that there exist primes $7<p < q_1 < r_1$ with $A(pq_1r_1) > \tfrac{2}{3}p$. By
Lemma \ref{chain} we can find primes $q$
and $r$ such that $A(pqr)\geq A(pq_1r_1)>\tfrac{2}{3}p$
and the chain of inequalities
\begin{equation} \label{inequ}
1 \leq p-q_p^* \leq r_p^*<p-r_p^* \leq q_p^* \leq p-1,
\end{equation}
is satisfied. We will assume $A(pqr) > \tfrac{2}{3}p$ 
and arrive at a contradiction.

First, by \tref{a_pqr}, there have to exist integers $i, j$ such that 
\begin{equation}
\label{2p3}
A(pqr)=|a_{pqr}(i)|=\left| \sum_{m \geq j} a_{pq}(m)\chi_i(m) \right| > \frac{2}{3}p.
\end{equation}
Our goal is now to show that this is impossible. The integers $i,j$ will be fixed during the whole proof. For brevity  we will denote 
$\chi_i(m)$ by  $\chi(m)$.

\par From \eqref{inequ}
it follows by \tref{Bzd} that 
$\alpha=p-q_p^*$, $\beta=p-r_p^*$ and
so $a_{pqr}(n)<r_p^*<p/2$ for any integer $n$. 
We infer that $a_{pqr}(i)$ is negative.
Therefore \tref{Bzd} yields $p-r_p^*\geq -a_{pqr}(i)>2p/3$, 
and thus 
\begin{equation}
\label{rp3}
r_p^* < \frac{p}{3},
\end{equation}
an inequality that  will play an important role in our proof. \par 
Since $a_{pqr}(i)$ is negative, we get from \eqref{2p3}
\begin{equation} \label{minus}
\sum_{m \geq j} a_{pq}(m)\chi(m)<-\frac{2}{3}p,
\end{equation}
and from \tref{sum0}
\begin{equation} \label{+}
\sum_{0\leq m < j} a_{pq}(m)\chi(m)>\frac{2}{3}p.
\end{equation}
\par If $a_{pq}(k)$ is non-zero, then we can 
write $k$ uniquely as either $up+vq$ or $up+vq-pq$ for some 
$0\le u\le q-1$ and $0 \leq v \leq p-1$ (see \lref{a_pq}).
\begin{definition} \label{u} (of $[\cdot ]$).
For any $v$, let $0\le [v] \le q-1$ be the unique integer such
that there exists an integer $s$ satisfying
$$([v]p+vq)r+q < i+1+spq \le ([v]p+vq)r+p+q.$$ 
\end{definition} 
\par It is not difficult to show that the set
$$\bigcup\limits_{u=0}^{q-1}\{(up+vq)r+1,\ldots,(up+vq)r+p\}$$
consists of $pq$ distinct numbers
and
forms a complete residue system modulo $pq$. 
From this the existence and uniqueness of $[v]$
follows. Observe that $v\equiv w\,(\mmod p)$ implies $[v]=[w]$.
\par Now we point out a helpful connection between $v$ and $v-r_p^*$.
\begin{lemma} \label{shift} Suppose that $0 \le u \le q-1$. Then
\begin{enumerate}
    \item[\normalfont a)] $\chi(up+vq)=1$ if and only if 
$u=[v]$.
    \item[\normalfont b)]  $\chi(up+vq)=-1$ if and only if $u=[v-r_p^*]$.
\end{enumerate}
\begin{proof}$~$
\begin{enumerate}
    \item[a)] This follows directly from Definitions \ref{chi} and \ref{u}.
    \item[b)] An easy consequence of the congruences 
\begin{align*}
([v-r_p^*]p+(v-r_p^*)q)r+q &\equiv ([v-r_p^*]p+(v-r_p^*)q)r+rr_p^*q \\ &\equiv  ([v-r_p^*]p+vq)r\,(\mmod{pq}),
\end{align*}
that establish a connection between the two intervals occurring in
Definition  \ref{chi}.\qedhere
\end{enumerate}
\end{proof}
\end{lemma}
We consider next when $a_{pq}(m)\chi(m)$ is non-zero. This will naturally lead to the consideration
of the following four sets:
\begin{align*}
C_{1,1}&=\{0\le v\le q_p^*-1:\,j\le [v-r_p^*]p+vq\text{~and~}[v-r_p^*] \le p_q^*-1\};\\
C_{1,2}&=\{q_p^* \le v\le  p-1:\, j \le [v]p+vq-pq\text{~and~}[v]\ge p_q^*\};\\
C_{2,1}&=\{0 \le v \le q_p^*-1:\,[v]p+vq < j
\text{~and~}[v] \le p_q^*-1\};\\
C_{2,2}&=\{q_p^* \le v \le p-1:\,[v-r_p^*]p+vq-pq <j\text{~and~}[v-r_p^*]\ge p_q^*\}.
\end{align*}
Let us look for example at the terms on the left-hand side of inequality  \eqref{minus} with $m \geq j$ and $a_{pq}(m)\chi(m) = -1$. 
By \lref{a_pq} we know that $m$ can be 
written as either $up+vq$ or $up+vq-pq$ with $0 \le v \le p-1$ and $0 \le u \le q-1$. 
If  $\chi(up+vq) \neq 0$ or $\chi(up+vq-pq)\neq 0$, then 
by \lref{shift} either $u=[v]$ 
or $u=[v-r_p^*]$. Simple considerations in this spirit then lead to the following lemma.
\begin{lemma}\label{fourcases}$~~$\\
{\rm a)} We have $\chi(m)=-1$, $a_{pq}(m)=1$ and
$m\ge j$ if and only of $m=[v-r_p^*]p+vq$ with 
$v\in C_{1,1}.$\\
{\rm b)} We have $\chi(m)=1$, $a_{pq}(m)=-1$ and
$m\ge j$ if and only if $m=[v]p+vq-pq$ with 
$v\in C_{1,2}.$\\
{\rm c)} We have $\chi(m)=1$, $a_{pq}(m)=1$ and
$m<j$ if and only if $m=[v]p+vq$ with $v\in C_{2,1}$.\\
{\rm d)} We have $\chi(m)=-1$, $a_{pq}(m)=-1$ and
$m<j$ if and only if $m=[v-r_p^*]p+vq-pq$ with 
$v\in C_{2,2}$.
\end{lemma}
An important property of the elements in the C-sets is whether they are \emph{low} or \emph{high}.
\begin{notation} \label{lowhigh} (Low, high).
If $0 \le v \le q_p^*-1$, we will call $v$ a \emph{low} integer. If $q_p^* \le v \le p-1$ we will call $v$ a \emph{high} integer. 
\end{notation}
Building on the $C$-sets we define some further ones. 
\begin{definition} \label{largedef}
Define
\begin{flalign*}
\phantom{0123}S~  =& \ (C_{1,1}\cap C_{2,1})\cup (C_{1,2}\cap C_{2,2})  &&\text{~(special integers)};&\\
\phantom{0123}P^+  =& \ \{v\in C_{1,1}:[v]\ge p_q^*\}\cup\{v\in C_{1,2}:[v-r_p^*] \le p_q^*\} && &\\
\phantom{0123}P^-  =& \ \{v\in C_{2,1}:[v-r_p^*]\ge p_q^*\} \cup  \{v\in C_{2,2}:[v] \le p_q^*\} && &\\
\phantom{0123}P  = & \ P^+\cup P^-  && \text{~(plain integers)};&\\
\phantom{0123}N  = &\ \{0,\ldots,p-1\} \setminus (P \cup S) && \text{~(null integers)}.&
\end{flalign*}
\end{definition}
The intersections $C_{1,1}\cap C_{2,1}$ and $C_{1,2}\cap C_{2,2}$ consist only of low, respectively
high integers. Note that the sets $S,P^+,P^-$ and $N$ are mutually disjoint.
\par We will show that a large value of $M(p)$ leads to a large value of $|S|-|N|$ (\lref{S-N}), which we subsequently show to be impossible.
\begin{lemma}\label{boundsp+} We have $|S|+|P^+|>\frac 23p$.
\begin{proof}
Assume $a_{pq}(m)\chi(m)=-1$ for some $m\geq j$. 
It follows
by Lemma \ref{fourcases} that $m$ satisfies 
either $m=[v-r_p^*]p+vq$ with 
$v\in C_{1,1}$ or 
$m=[v]p+vq-pq$ with $v\in C_{1,2}$. 
Assume first $v\in C_{1,1}$. 
If $[v]\geq p_q^*$, then $v\in P^+$. If $[v]\leq p_q^*-1$ and $[v]p+vq<j$, then $v\in C_{2,1}$ 
and so $v\in C_{1,1}\cap C_{2,1}\subseteq S$.
Hence, $v\notin S\cup P^+$ implies $[v]\leq p_q^*-1$ and $j\leq [v]p+vq$ and thus $a_{pq}(m')\chi(m')=1$ for $m'=[v]p+vq$. It follows that $a_{pq}(m)\chi(m)+a_{pq}(m')\chi(m')=0$ and
so the terms corresponding to $m$ and $m'$  in
\eqref{minus} cancel.
\par Next assume $m=[v]p+vq-pq$ and 
$v\in C_{1,2}$. Then, by
a similar reasoning, either $v\in S\cup P^+$ or both $p_q^*\leq [v-r_p^*]$ and $j\leq m'$ with
$m'=[v-r_p^*]p+vq-pq$. 
Again, the terms corresponding to $m$ and $m'$ in
\eqref{minus} cancel.

In order for the sum in \eqref{minus} to be smaller than $-2p/3$, there must be more than $2p/3$ terms with value $-1$ which are not canceled like this. 
As we have just seen, these 
satisfy $v\in S\cup P^+$. Since $S$ and $P^+$ are disjoint, the proof is completed.
\end{proof}
\end{lemma}
\begin{lemma}\label{boundsp-} We have $|S|+|P^-|>\frac 23p$.
\begin{proof}
The proof is similar to the proof of \lref{boundsp+}. Assume $a_{pq}(m)\chi(m)=1$ for some $m<j$, 
then by Lemma \ref{fourcases} 
either $v\in C_{2,1}$ or $v\in C_{2,2}$. 
If $v\in C_{2,1}$, then $m=[v]p+vq$ and either $v\in S\cup P^-$ or $a_{pq}(m')\chi(m')=-1$ for $m'=[v-r_p^*]p+vq$ 
with $m'<j$. 

Similarly, if $v\in C_{2,2}$, then $m=[v-r_p^*]p+vq-pq$ and either $v\in S\cup P^-$ or $a_{pq}(m')\chi(m')=-1$ for $m'=[v]p+vq-pq$ with $m'<j$. 

In both cases, the terms $a_{pq}(m)\chi(m)=1$ cancel out against the terms $a_{pq}(m')\chi(m')=-1$ if $v\notin S\cup P^-$. By $\eqref{+}$ and the disjointness 
of $S$ and $P^-$ it then follows that $|S|+|P^-|>\frac 23p$.\end{proof}

\end{lemma}
\begin{lemma} \label{S-N}
If $A(pqr)>\frac 23 p$, then
\begin{equation}
\label{SminN}
|S|-|N|>\frac{p}{3}.
\end{equation}
\begin{proof}
Since by definition each $v$ is in exactly one of the 
mutually disjoint sets $S, P^+, P^-$ and $N$, we have
\begin{equation} \label{whole}
|S|+|P^+|+|P^-|+|N|=p.
\end{equation}
On combining this identity with \lref{boundsp+}, \lref{boundsp-} and \eqref{whole} the proof is completed. 
\end{proof}
\end{lemma}
\begin{lemma}
\label{lowineq}
If $v$ is both low and special, then $[v]<[v-r_p^*]\leq p_q^*-1$.
\begin{proof}
By \dref{largedef} ($S$) we have
$v\in C_{1,1}\cap C_{2,1}$ and so $[v-r_p^*]\leq p_q^*-1$ and $[v]p+vq<j\leq [v-r_p^*]p+vq$, as desired. 
\end{proof}
\end{lemma}
\begin{lemma}
\label{highineq}
If $v$ is both high and special, then $p_q^*\leq [v-r_p^*]<[v]$.
\begin{proof}
By \dref{largedef} ($S$) we have $v\in C_{1,2}\cap C_{2,2}$. 
It follows that $[v-r_p^*]\geq p_q^*$ and $[v-r_p^*]p+vq-pq<j\leq [v]p+vq-pq$, as desired. 
\end{proof}
\end{lemma}
In order to prove Theorem \ref{cbc} we 
will derive a contradiction to the conclusion of \lref{S-N}. In principle we need to show that there can't be much more special than null integers. First we will prove a lemma which gives some bounds for integers which can be special. Afterwards, we will define an injection $f$ from a subset of $S$ into $N$ and a bijection $g$ between two distinct sets of integers where only one integer can be special. The domain and range from $f$ and $g$ will be disjoint. Therefore, the elements in the domain of $f$ are not important for the difference $|S|-|N|$ and we will get $|S|-|N| \leq p/3$ because of the bounds for special integers and the bijection $g$. First, we bound the low special integers.
As there are  $p-q_p^*$ 
integers in $[q_p^*,p-1]$, by  
\eqref{inequ}, \eqref{rp3} and 
\eqref{SminN} we find
$p-q_p^*\le r_p^*<\frac p3< |S|-|N|<|S|$, 
and so there is at least one low special integer. 
\begin{notation} 
\label{defv0} ($v_0$).
Let $v_0$
denote the largest low special integer.
\end{notation}
\begin{lemma} \label{rp*+}
If $v \in S\cup P^-$ is a low integer, 
then $v_1 \notin S\cup P^+$ for every $0 \leq v_1 \leq v-p+q_p^*$. 
\begin{proof}
The assumption on $v$ implies by \dref{largedef} ($S$) and \dref{largedef} ($P^-$), that $v\in C_{2,1}$. Hence, $[v]\leq p_q^*-1$ and $[v]p+vq<j$.

Assume for the sake of contradiction that $v_1\in S\cup P^+$ for some such $v_1$. As $v_1\leq v-p+q_p^*\leq v\leq q_p^*-1$, the integer $v_1$ is low. By \dref{largedef} ($S$) and \dref{largedef} ($P^+$) it now follows that 
$v_1\in C_{1,1}$ and hence $[v_1-r_p^*] \le  p_q^*-1$ and $j \le [v_1-r_p^*]p+v_1q$. We 
conclude that $[v]p+vq<j\leq [v_1-r_p^*]p+v_1q$. 
{}From this it follows that 
$(v-v_1)q<([v_1-r_p^*]-[v])p,$ and hence
\begin{equation*}
(p-q_p^*)q \le (v-v_1)q < ([v_1-r_p^*]-[v])p.
\end{equation*}
Using \lref{pqinv} this leads to a contradiction, since we also have
\begin{equation}
([v_1-r_p^*]-[v])p\leq [v_1-r_p^*]p \leq (p_q^*-1)p = (p-q_p^*)q-p+1,
\end{equation}
completing the proof.
\end{proof}
\end{lemma}
\begin{corollary} \label{bound}
If an integer $v$ is both low and special, 
then $v_0-p+q_p^*+1\le v\le v_0$. 
\end{corollary}
From this we deduce that the number of low special integers is at most $p-q_p^*\leq r_p^*<\frac p3<|S|$, so there must exist high special integers.
\begin{lemma} \label{rp*-}
If $v \in S\cup P^+$ is a low integer, we have $v_1 \notin S\cup P^-$ for every $v+r_p^* \leq v_1 \leq q_p^*-1$. \begin{proof}
Similar to the proof of \lref{rp*+}.
The assumption on $v$ implies  by \dref{largedef} ($S$) and \dref{largedef} ($P^+$) that $v\in C_{1,1}$ 
and hence $[v-r_p^*]\leq p_q^*-1$ and
$j\le [v-r_p^*]p+vq$.

Assume for the sake of contradiction that $v_1\in S\cup P^-$ for such $v_1$. Since $v_1$ is low by 
assumption, we have 
$v_1\in C_{2,1}$ by \dref{largedef} ($S$) and \dref{largedef} ($P^-$).
Thus $[v_1]\leq p_q^*-1$ and $[v_1]p+v_1q<j$. We 
conclude that  $[v_1]p+v_1q<j\leq [v-r_p^*]p+vq$. 
{}From this it follows that
$(v_1-v)q<([v-r_p^*]-[v_1])p$, which implies \[(p-q_p^*)q\leq r_p^*q\leq (v_1-v)q<([v-r_p^*]-[v_1])p.\] Using \lref{pqinv} this leads to a contradiction, since we also have \[([v-r_p^*]-[v_1])p\leq [v-r_p^*]p\leq(p_q^*-1)p=(p-q_p^*)q+1-p,\] completing the proof.
\end{proof}
\end{lemma}
\par 
If $v\in S\cup P$ (recall
Definition \ref{largedef}) 
we can relate $[v]$ and $[v-r_p^*]$. 
Our goal, however, is to obtain a contradiction to
the conclusion of \lref{S-N}. 
This  can only be reached with further information about $[v]$ and $[v-r_p^*]$ 
and their difference.  
For ease of notation we make the
following definition.
\begin{definition}
\label{hqdef}
(of $h(v)$ and $h_q(v)$).
We put 
$$h(v)=[v]-[v-r_p^*].$$
The unique integer $0\le \beta<q$ such that 
$h(v)\equiv \beta\,(\mmod{q})$
we denote by $h_q(v)$.
\end{definition}
We will need a technical lemma. 
\begin{lemma} \label{modchi} 
Given any integer $0 \le v \le p-1$, one of the following holds true.\\

\noindent {\rm (a)} There exists some $\overline{q_p}\le k\le p-1$ such that
\begin{equation}
    \label{longq_p}
([v]p+vq)r+p+q \equiv i+1+k \equiv ([v-r_p^*]p+vq)r+p+\overline{q_p} \,(\mmod{pq}),
\end{equation}
which implies that
\begin{equation}
\label{q_p}
h(v)\,pr \equiv \overline{q_p}-q \,(\mmod{pq}).
\end{equation}
{\rm (b)} There exists some 
$0\le k\le \overline{q_p}-1$ such that
\begin{equation}
\label{longp-q_p}
([v]p+vq)r+p+q \equiv i+1+k \equiv ([v-r_p^*]p+vq)r+p+(\overline{q_p}-p)\; (\mmod{pq}),
\end{equation}
which implies that
\begin{equation}
\label{p-q_p}
h(v)\,pr \equiv \overline{q_p}-q-p \,(\mmod{pq}).
\end{equation}
\begin{proof}
By \dref{u}, we have 
\begin{equation} \label{mequiv}
([v]p+vq)r+p+q \equiv i+1+k_1 \,(\mmod{pq})\text{~with~}k_1\in I_p,
\end{equation} 
By the same definition, but this time applied to $v-r_p^*$, we
obtain 
\begin{equation} \label{nequiv}
([v-r_p^*]p+vq)r+p \equiv i+1+k_2 \,(\mmod{pq})\text{~with~}k_2\in I_p,
\end{equation} 
on noting that $r_p^*rq\equiv q(\mmod{pq})$.
Therefore, we have 
\begin{equation} \label{m-nequiv}
h(v)\,pr+q \equiv k_1-k_2 \,(\mmod{pq}).
\end{equation}
We infer that $q \equiv k_1-k_2 \,(\mmod{p})$.
Note that $k_1-k_2\in [-(p-1),\; p-1]$, an interval 
of length smaller than $2p$. Therefore 
either  $k_1-k_2=\overline{q_p}$ or  $k_1-k_2=\overline{q_p}-p$.
From (\ref{mequiv}), (\ref{nequiv}) and the new information for (\ref{m-nequiv}) we get (\ref{q_p}), respectively  (\ref{p-q_p}). In case $k_1-k_2=\overline{q_p}$ we have $k_1 \geq \overline{q_p}$. In case 
$k_1-k_2=\overline{q_p}-p$, we have $k_1 \leq \overline{q_p}-1$. 
On taking $k=k_1$ the proof is completed.
\end{proof}
\end{lemma}
\begin{corollary} \label{2modq}
For any integer $0 \le v \le p-1$, there are two possible remainders modulo $q$ for the 
difference $h(v)=[v]-[v-r_p^*]$, that is $\left|\,\{h_q(v):0\le v\le p-1\}\,\right|\le 2$.
\end{corollary} 
This is a simple 
consequence of \lref{modchi}. Either 
$h(v)$ satisfies the congruence
\eqref{q_p} or the congruence \eqref{p-q_p} and each of
these equations has an 
unique solution $h(v)$ modulo $q$. Note that
the two solutions are distinct.
\par Since there is at least one low special and one high special integer, the next lemma in combination with
Corollary \ref{2modq} shows that 
$$\left|\,\{h_q(v):0\le v\le p-1\}\,\right|=2.$$ 
\begin{lemma}\label{chiclass}
Let $v$ be a low special integer and $v'$ be a high special integer. 
Then $h_q(v)\ne h_q(v')$.
\begin{proof}
Assume the two terms are congruent. 
By \lref{lowineq} and \lref{highineq}, $[v]<[v-r_p^*]$ and $[v'-r_p^*]<[v']$. Hence, $0<[v-r_p^*]+[v']-[v]-[v'-r_p^*]$, which is divisible by $q$ by assumption and thus at least $q$. On the other hand, since $v\in C_{1,1}$ and $v'\in C_{2,2}$, we also have
$$h(v')-h(v)=[v-r_p^*]+[v']-[v]-[v'-r_p^*] \leq (p_q^*-1)+(q-1)-0-p_q^* = q-2 <q,$$
which is a contradiction. 
\end{proof}
\end{lemma}

\begin{lemma} \label{sameclass}
Let $v_1,v_2$ be two special integers. Then
$h(v_1)=h(v_2)$ if and only if 
$v_1$ and $v_2$ are both low or both high integers. 
\end{lemma}
\begin{proof} The weaker result with identity replaced by being congruent modulo $q$ follows from 
Lemma \ref{chiclass} and Corollary \ref{2modq}.

For low special integers $v_1$ and $v_2$ we
have  $-q<h(v_i)<0$ by \lref{lowineq}. 
Since the two 
expressions for $i=1$ and $i=2$ are congruent mod $q$, they have to be equal. 
For high special integers $v_1$ and $v_2$ we have
$0<h(v_i)<q$ by \lref{highineq}. 
Again the two expressions are congruent mod $q$ and thus equal.
\end{proof}
\begin{lemma} \label{s-shift-p}
	Suppose that $v' \in S \cup P^-$ is high 
	and $v'-r_p^* \in P$. 
	Let $v_1$ be a low special integer.
	Then 
	$v'-r_p^* \in P^+$ and
	$h_q(v'-r_p^*)\ne h_q(v_1)$. 
\end{lemma}
\begin{proof}
	 By \dref{largedef} ($S$) and ($P^-$), and since $v'$ is high, we have  $v'\in C_{2,2}$ and hence $[v'-r_p^*]p+(v'-p)q < j$ and $p_q^*\le [v'-r_p^*]$. Thus from $v'-r_p^*<p-r_p^* \le q_p^*$ being a low integer, we infer that $v'-r_p^* \not\in P^-$ (since $v'-r_p^* \in P^-$ would imply $v\in C_{2,1}$ contradicting $p_q^*\leq [v'-r_p^*]$). Therefore, $v'-r_p^* \in P^+$ and so $v'-r_p^*\in C_{1,1}$, implying $[v'-2r_p^*]\leq p_q^*-1<[v'-r_p^*]$. Since  $0\leq [v_1]<[v_1-r_p^*]\leq p_q^*-1$ by \lref{lowineq}, we conclude that $[v_1-r_p^*]-[v'-2r_p^*]-[v_1]+[v'-r_p^*]>0$.
	 
	 Since $v'-r_p^*\in C_{1,1}$ and $v'\in C_{2,2}$, we have $[v'-r_p^*]p+v'q-pq<j\leq [v'-2r_p^*]p+(v'-r_p^*)q$, and so $p([v'-r_p^*]-[v'-2r_p^*])<(p-r_p^*)q$.
	 Combining the various inequalities, we deduce 
	\begin{align*} 
	0&<([v_1-r_p^*]-[v_1]-[v'-2r_p^*]+[v'-r_p^*])p < (p_q^*-1)p + (p-r_p^*)q \\ &\leq p(p_q^*-1)+qq_p^* = pq-p+1<pq.
	\end{align*}
	Hence,  the numbers $h(v_1)$ and $h(v'-r_p^*)$ 
	cannot be congruent modulo $q$, as their difference 
	lies in the interval $[1,q-1]$.
\end{proof}
\begin{lemma} \label{s-shift-p.c}
Suppose that $v'\in S$ is high and $v'-r_p^*\notin N$. Then $v'-r_p^*\in P^+$ and $h(v')= h(v'-r_p^*)$. Moreover, $h_q(v')=h(v')$ and  
$h_q(v'-r_p^*)=h(v'-r_p^*)$.
\begin{proof}
As $v'-r_p^*<p-r_p^*\leq q_p^*$, we conclude that $v'-r_p^*$ is low.

Assume first that $v'-r_p^*\in S$. By \dref{largedef} ($S$), $v'-r_p^*\in C_{2,1}$, so $[v'-r_p^*]\leq p_q^*-1$. Since $v'$ is high and special, $v'\in C_{2,2}$, so 
$[v'-r_p^*]\ge p_q^*$, a contradiction.

Hence, $v'-r_p^*\notin S$, so $v'-r_p^*\in P$. By \lref{s-shift-p} and \cref{2modq}, we know $v'-r_p^*\in P^+$ with $h(v'-r_p^*)=h(v')~(\mmod{q})$. We now show 
that these two numbers are equal. We have $h(v')\in [0,q-1]$ by \lref{highineq}. 
Further, $[v'-2r_p^*]\leq p_q^*-1<[v'-r_p^*]$ 
by \dref{largedef} ($P^+$), and so also $h(v'-r_p^*)\in [0,q-1]$. The final
statement follows by \dref{hqdef}.
\end{proof}
\end{lemma}
\begin{lemma} \label{s-shift-p.2}
Suppose that $v' \in S \cup P^+$ is high, and 
$v'+r_p^*-p \in P$. Then 
$v'+r_p^*-p \in P^-$ and
$h_q(v'+r_p^*-p)\ne h_q(v_0)$. 
\begin{proof}
The proof is similar to that given in \lref{s-shift-p}. Since $v'\in S\cup P^+$ is high, $v'\in C_{1,2}$, so $[v']\geq p_q^*$ and $[v']p+v'q-pq\geq j$. Since $v'+r_p^*-p\geq v'+p-q_p^*-p\geq 0$ and $v'+r_p^*-p\leq p-1+r_p^*-p<q_p^*$, we
conclude that $v'+r_p^*-p$ is low.

Assume for a contradiction that $v'+r_p^*-p\in P^+$. Then 
$v'+r_p^*-p\in C_{1,1}$ by \dref{largedef} ($P^+$). Thus, $[(v'+r_p^*-p)-r_p^*]=[v'-p]=[v']\le p_q^*-1$, a contradiction. Hence, $v'+r_p^*-p\in P^-$.
\par By \dref{largedef} ($P^-$), $v'+r_p^*-p\in C_{2,1}$. Hence, $[v'+r_p^*-p]p+(v'+r_p^*-p)q<j\leq [v']p+v'q-pq$, or, equivalently, $p([v']-[v'+r_p^*-p])>r_p^*q$. Since $[v_0]<[v_0-r_p^*]$ and $[v'+r_p^*-p]\le p_q^*-1<[v']$ by \lref{lowineq}, it follows that 
both $[v_0]-[v_0-r_p^*]$ and $[v'+r_p^*-p]-[v']$ lie in the interval $[-q+1,-1]$. Thus, if they are congruent modulo $q$, they have to be equal. However, 
$$p([v_0-r_p^*]-[v_0]+[v'+r_p^*-p]-[v'])<p(p_q^*-1)-r_p^*q,$$ and the proof is
finished on noticing that, by \lref{pqinv} and \eqref{inequ}, 
\[p(p_q^*-1)-r_p^*q=
pq+1-qq_p^*-p-r_p^*q<q(p-q_p^*-r_p^*)\leq 0.\qedhere\]
\end{proof}
\end{lemma}
\begin{lemma} \label{+shift}
Assume $r_p^*+q_p^*=p$. If $q_p^* \le v' \le p-1$ and $v' \in P$ with $v'+r_p^*-p \in S$, then $h_q(v_0)=h_q(v')$.
\begin{proof}
As in the previous lemma, $v'+r_p^*-p\leq p-1+r_p^*-p<q_p^*$, so $v'+r_p^*-p$ is low. By \dref{largedef} ($S$), $v'+r_p^*-p\in C_{1,1}$, 
so $[(v'+r_p^*-p)-r_p^*]=[v'-p]=[v']\le p_q^*-1$. If $v'\in P^+$, 
then $v'\in C_{1,2}$ by \dref{largedef} ($P^+$), and hence $[v']\ge p_q^*$, contradicting the above inequality. Hence, $v'\in P^-$ and $v'\in C_{2,2}$, so $[v'-r_p^*]p+v'q-pq<j$. Since $v'+r_p^*-p\in C_{1,1}$, we have $j\leq [v'-p]p+(v'+r_p^*-p)q=[v']p+(v'+r_p^*-p)q$. Combining these inequalities for $j$, we deduce $[v'-r_p^*]p+v'q-pq<[v']p+(v'+r_p^*-p)q$, 
which simplifies to $p\,h(v')>-r_p^*q$.
By \dref{largedef} ($P^-$) and the definition of $C_{2,2}$, $v'\in C_{2,2}$ being high implies that $[v']\leq p_q^*\leq [v'-r_p^*]$, and so $-q<h(v')\leq 0$.

Let $v$ be a high special integer. If $h_q(v_0)\ne h_q(v')$, then by 
Corollary \ref{2modq} and \lref{sameclass}, we infer
that $h_q(v)=h_q(v')$. By \lref{highineq}, $0<h(v)<q$ and so the only possibility is $h(v)-q=h(v')$. However, then $-r_p^*q<p\,h(v')=p(h(v)-q)$. Since $v$ is high and special, we have $[v-r_p^*]\geq p_q^*$ as $v\in C_{2,2}$, and so $-r_p^*q<p(q-1-p_q^*-q)$.
Since by assumption $r_p^*+q_p^*=p$, the
latter inequality can 
be rewritten as $-(p-q_p^*)q<p(-1-p_q^*)$. 
This on its turn can be rewritten as
$qq_p^*+pp_q^*<pq-p$. As the left-hand side
equals $1+pq$ we have obtained a contradiction.
\end{proof}
\end{lemma}
\begin{definition} \label{f} (of the map $f$).
Define $f: I_p \rightarrow I_p$ by
$$f(v) = 
\begin{cases}
v-r_p^* & \text{if } v \geq r_p^* ; \\
v-r_p^* + p & \text{if } v \leq r_p^*-1.
\end{cases} $$
Obviously, $f$ is a bijection.
\end{definition}
\begin{lemma} \label{shiftclass}
If $v \in S$, then $f(v) \notin S$.
\begin{proof}
For a contradiction assume that $f(v)$ is special.
\par If $f(v)$ and $v$ are both low, 
then by  \cref{bound} either both $v$ and $v-r_p^*$ are in the interval
$[v_0-p+q_p^*+1,v_0]$, or both $v$ and $v-r_p^*-p$ are 
in this interval. In the first case 
we must have $r_p^*<p-q_p^*$, and in the second case $p-r_p^*<p-q_p^*$. 
Both of these inequalities contradict \eqref{inequ}.
\par The integers
$v$ and $f(v)$ cannot be both high integers, since 
then their difference is at most $p-q_p^*-1$, which by
\eqref{inequ} is both smaller than $r_p^*$ and $p-r_p^*$.
\par If one of $v$ and $f(v)$ is high and 
the other low,  
then  by \dref{largedef} ($S$) it would 
follow that $p_q^* \le [v-r_p^*] \le p_q^*-1$, which
is impossible.
\end{proof}
\end{lemma}
\begin{definition} \label{S0}
(of $S_0$). Let $S_0$ be the set of integers $v \in S$ with $f(v) \in P$.
\end{definition}
For $v \in S \setminus S_0$ we have 
$f(v)\not\in P$ (by definition) and 
$f(v)\not\in S$ (by Lemma \ref{shiftclass}), and hence
$f(v) \in N$. 
\par Now we need pairs of integers from which only one can be in $S_0$. 
\begin{definition} \label{g}
(of the map $g$).
For any integer $v\in [0,p-1]$, we define the function $g$ by
$$g(v) = 2v_0+r_p^*-v.$$
\end{definition}
Note that a priori we might have $g(v)<0$, but the 
corresponding values of $v$ will not play a role in our arguments.
\begin{lemma} \label{2v0+rp*-v}
If $v \in S_0$ and $v \geq q_p^*$, and if also $0 \le 2v_0+r_p^*-v \le v_0$, then $g(v)=2v_0+r_p^*-v$ is
not in $S \cup P^+$.
\begin{proof}
\indent Assume for a contradiction that $g(v)\in S\cup P^+$. Either \eqref{longp-q_p} or \eqref{longq_p} is true for $v_0$. Assume first \eqref{longp-q_p} holds for $v_0$. Then,
\begin{itemize}
	\item $([v_0]p+v_0q)r+p+q \equiv i+1+k_1 \,(\mmod{pq})$ and
	\item $([v_0-r_p^*]p+v_0q)r + \overline{q_p} \equiv i+1+k_1 \,(\mmod{pq})$ with $0\le k_1\le \overline{q_p}-1$.
\end{itemize}
In particular, 
\begin{equation}
\label{redmodp}
v_0qr+q\equiv i+1+k_1\,(\mmod{p}).
\end{equation}
Furthermore, by \dref{u} and noting that $r_p^*rq\equiv q\,(\mmod{pq})$,
\begin{itemize}
	\item $([2v_0-v]p+(2v_0+r_p^*-v)q)r+p \equiv i+1+k_2 \,(\mmod{pq})$ with $0\le k_2\le p-1$.
\end{itemize}
Since by assumption $v\ge q_p^*$, it follows that
$f(v)=v-r_p^*$ as $q_p^*\ge r_p^*$. Since by assumption
$v\in S_0,$ we infer that $f(v)=v-r_p^*\in P$ by the
definition of $S_0$. 
As we assumed that \eqref{p-q_p} holds, it follows by
\lref{s-shift-p} (with $v'=v$ and $v_1=v_0$) that 
\eqref{q_p} (and so \eqref{longq_p}) holds for $v-r_p^*$, and hence
\begin{itemize}
	\item $([v-r_p^*]p+(v-r_p^*)q)r+p+q \equiv i+1+k_3\,(\mmod{pq})$ with $\overline{q_p}\le k_3\le p-1$.
\end{itemize}
Similarly \eqref{longq_p} holds by \lref{sameclass} 
applied to $v$ (which is high and special), giving
\begin{itemize}
	\item $([v-r_p^*]p+vq)r+p+\overline{q_p} \equiv i+1+k_4 \,(\mmod{pq})$ with $\overline{q_p}\le k_4\le p-1$.
\end{itemize}
The last two congruences imply that $q$ divides $k_3-k_4+\overline{q_p}$. 
Clearly $k_3-k_4+\overline{q_p}\in (-p,p)\subseteq (-q,q)$, 
hence $k_3=k_4-\overline{q_p}$ and $k_3 \in [\overline{q_p}, p-1-\overline{q_p}]$. 
\par From the congruences involving 
$k_2$ and $k_3$ and \eqref{redmodp} we get 
\begin{eqnarray*}
2(i+1)+k_3+k_2  & \equiv & ([v-r_p^*]p+(v-r_p^*)q)r+p+q+([2v_0-v]p+(2v_0+r_p^*-v)q)r+p \\ 
& \equiv & 2v_0qr+q\equiv 2(i+1)+2k_1-\overline{q_p}\,(\mmod{p}),
\end{eqnarray*}
showing that $p$ divides $k_3+k_2-2k_1+\overline{q_p}$. 
Since $k_3+k_2-2k_1+\overline{q_p} \leq p-\overline{q_p}-1+p-1  + \overline{q_p} < 2p$ and 
$0<\overline{q_p}-2(\overline{q_p}-1)+\overline{q_p} \le k_3+k_2-2k_1+\overline{q_p}$, it follows that
\begin{equation*}
k_3+k_2-2k_1+\overline{q_p}=p.
\end{equation*}
Using this, we conclude that
\begin{eqnarray*}
 & & ([v-r_p^*]p+(v-r_p^*)q)r+p+q+[2v_0-v]p+(2v_0+r_p^*-v)q)r+p \\ 
& \equiv & i+1+k_3+i+1+k_2 \equiv i+1+k_1 + p  -\overline{q_p} + i+1+k_1 \\ 
& \equiv & ([v_0-r_p^*]p+v_0q)r+\overline{q_p}+p-\overline{q_p}+([v_0]p+v_0q)r+p+q \,(\mmod{pq}), 
\end{eqnarray*}
and thus $q$ divides $[v-r_p^*]+[2v_0-v]-[v_0]-[v_0-r_p^*]$. 
Since, by \dref{largedef} ($S$),  $v_0\in C_{1,1}$ and $v\in C_{2,2}$, 
we have  $[v_0-r_p^*] < p_q^* \le [v-r_p^*]$. 
By assumption, $g(v)\in S\cup P^+$ is low, so $2v_0-v+r_p^*\in C_{1,1}$. Hence, on 
also noting that $v_0\in C_{2,1}$, we see that 
$[v_0]p+v_0q < j \le [2v_0-v]p+(2v_0-v+r_p^*)q < [2v_0-v]p+v_0q$ 
(since $g(v)\leq v_0$ implies $v_0\leq v-r_p^*$), so $[v_0]<[2v_0-v]$. Thus $q$
divides the positive number $[v-r_p^*]+[2v_0-v]-[v_0]-[v_0-r_p^*]$, and so this 
number is $\ge q$. From this we infer that
\begin{equation}
    \label{nogeenongelijkheid}
2[v-r_p^*]\ge 2(q+[v_0-r_p^*]+[v_0]-[2v_0-v])\ge 2(q-p_q^*+1),
\end{equation}
where the final inequality follows on noting that
$[2v_0-v]\le p_q^*-1$ (a consequence of the
fact that $2v_0-v+r_p^*\in C_{1,1}$).
On the other hand, by applying \lref{s-shift-p.c} to $v$ (which is allowed since $v\in S_0$ implies $f(v)\notin N$), we get
\begin{equation} \label{2u_2'}
2[v-r_p^*]= [v]+[v-2r_p^*] \leq q-1+p_q^*.
\end{equation}
On combining \eqref{nogeenongelijkheid} and
\eqref{2u_2'}, we obtain $2(q-p^*_q+1)\leq q-1+p^*_q$, 
which on multiplying both sides by $p$ gives rise to
 $pq+3p\leq 3pp^*_q$. Since $2p/3<p-r_p^*\leq q_p^*$, 
it now follows on invoking \lref{pqinv} that
 $pq+3p\le 3pp_q^*=3(pq+1-qq_p^*)<pq+3$, which is
impossible.

The proof in the case that \eqref{q_p} holds for $v_0$ is analogous
and now the congruences above with a bullet point get replaced, respectively, by
\begin{itemize}
    \item $([v_0]p+v_0q)r+p+q \equiv i+1+k_1 \,(\mmod{pq})$ and
	\item $([v_0-r_p^*]p+v_0q)r +p+ \overline{q_p} \equiv i+1+k_1 \,(\mmod{pq})$ with $\overline{q_p}\le k_1\le p-1$,
	\item $([2v_0-v]p+(2v_0+r_p^*-v)q)r+p \equiv i+1+k_2 \,(\mmod{pq})$ with $0\le k_2\le p-1$,

	\item $([v-r_p^*]p+(v-r_p^*)q)r+p+q \equiv i+1+k_3\,(\mmod{pq})$ with $0\leq k_3\leq\overline{q_p}-1$,
    \item $([v-r_p^*]p+vq)r+\overline{q_p} \equiv i+1+k_4 \,(\mmod{pq})$ with $0\leq k_4\leq\overline{q_p}-1$.
\end{itemize}
 Subtracting the latter congruence from
 the previous, we 
 conclude that $q$ divides $k_3-k_4-p+\overline{q_p}$. Since $k_3-k_4-p+\overline{q_p}\leq\overline{q_p}-1-p+\overline{q_p}<2p-p<q$ and $k_3-k_4-p+\overline{q_p}\geq -\overline{q_p}+1-p+\overline{q_p}>-q$,
 we infer that $k_3-k_4-p+\overline{q_p}=0$. As $k_4=k_3-p+\overline{q_p}\in [0,\overline{q_p}-1]$, it follows that $k_3\in [p-\overline{q_p},\overline{q_p}-1]$.
Exactly as above, we conclude that $p$ divides $k_3+k_2-2k_1+\overline{q_p}$. Since also \[-p<p-\overline{q_p}-2(p-1)+\overline{q_p}\leq k_3+k_2-2k_1+\overline{q_p}\leq\overline{q_p}-1+p-1-2\overline{q_p}+\overline{q_p}<p,\] we obtain $k_3+k_2-2k_1+\overline{q_p}=0$. As before, this implies
\begin{eqnarray*}
 &  & ([v-r_p^*]p+(v-r_p^*)q)r+p+q+([2v_0-v]p+(2v_0+r_p^*-v)q)r+p \\ 
& \equiv & i+1+k_3+i+1+k_2 \equiv i+1+k_1  -\overline{q_p} + i+1+k_1 \\ 
& \equiv & ([v_0-r_p^*]p+v_0q)r+p+\overline{q_p}-\overline{q_p}+([v_0]p+v_0q)r+p+q \,(\mmod{pq}), 
\end{eqnarray*}
and thus $q$ divides  $[v-r_p^*]+[2v_0-v]-[v_0]-[v_0-r_p^*]$. We conclude in exactly the same way as above.
\end{proof}
\end{lemma}
We will now finish the proof by distinguishing three cases depending on the value of $v_0$.\\
\\
\textbf{Case 1:} $0 \leq v_0 \leq r_p^*-1$. \\
In this case, all low special integers lie in the interval $[0,r_p^*-1]$. Additionally, all high special integers lie in the interval $[q_p^*,p-1]\subseteq[p-r_p^*,p-1]$. 
By \lref{shiftclass}, 
for every $v\in[0,r_p^*-1]$ 
at most one of $v$ and $v-r_p^*+p$ is special.
Thus, 
\begin{equation*}
|S| \leq r_p^* < \frac{p}{3}.
\end{equation*}
\\
\textbf{Case 2:} $r_p^* \leq v_0 \leq q_p^*-r_p^*-1$. \\
We structure this lengthy case by formulating four claims
in the proof. 
\begin{claim}
If $v\in S$ and $v<v_0$, then $g(v)\in [q_p^*,p-1]\setminus S$ or $\{2v_0-v-r_p^*,2v_0-v,2v_0-v+r_p^*\}$ shares an element with $N$. 
\begin{proof}
By \cref{bound}, we know that $v\in [v_0-p+q_p^*+1,v_0]$ for any low special $v$. We will distinguish two ranges of $v$.
\par If $v \in [v_0-p+q_p^*+1,2v_0+r_p^*-q_p^*]$, then $g(v)=2v_0-v+r_p^*\in [q_p^*,p-1]$ (where
we use the assumption $v_0\le q_p^*-r_p^*-1$). Since $2v_0+r_p^*-q_p^*<v_0$ by the conditions of Case $2$, it follows from \lref{2v0+rp*-v} that $2v_0-v+r_p^* \notin S_0$. 
Applying \lref{2v0+rp*-v} to $2v_0+r_p^*-v$, we conclude that
$2v_0+r_p^*-v\not\in S_0$.
Thus, we have $g(v)=2v_0-v+r_p^*\in [q_p^*,p-1]\setminus S$ or $f(2v_0-v+r_p^*)=2v_0-v\in N$ (cf.\,with the sentence just
below \dref{S0}).
\par If $v \in [2v_0+r_p^*-q_p^*+1,v_0-1]$, we know that $2v_0-v \notin S$ because $v_0+1 \leq 2v_0-v \leq q_p^*-r_p^*-1<q_p^*$, so $2v_0-v$ is larger than all low special integers and smaller than all high integers. 
Furthermore, $2v_0-v\ge q_p^*-r_p^*-1\ge p-1-2r_p^*\ge r_p^*$.
\par If $2v_0-v\in N$ there is nothing to prove and, as $2v_0-v\notin S$,  it remains to consider the case where $2v_0-v\in P$.
    \begin{itemize}
        \item If $2v_0-v \in P^+$, then $[2v_0-v]\ge p_q^*$ by \dref{largedef} ($P^+$). Note that $2v_0-v+r_p^*$ is low.  We have
        $2v_0-v+r_p^*\notin P^+$ (as otherwise 
        $2v_0-v+r_p^*\in C_{1,1}$ by \dref{largedef} ($P^+$) 
        and hence $[2v_0-v]<p_q^*$). Since $2v_0-v \in P^+$ is low, we conclude by \lref{rp*-} with $v_1=2v_0-v+r_p^*$ that $v_1\notin P^- \cup S$.
        All in all, $2v_0-v+r_p^* \in N$. 
        \item If $2v_0-v \in P^-$, then $[2v_0-v-r_p^*]\ge p_q^*$ by \dref{largedef} ($P^-$). Note that $2v_0-v-r_p^*$ is low. 
        We have $2v_0-v-r_p^* \notin P^-$ (as otherwise $2v_0-v-r_p^*\in C_{2,1}$ by \dref{largedef} ($P^-$) and hence
        $[2v_0-v-r_p^*]<p_q^*$). Since $2v_0-v \in P^-$ is low, we conclude by \lref{rp*+} with $v_1=2v_0-v-r_p^*$ that $v_1\notin P^+ \cup S$ (note that $0\le v_1\le 2v_0-v-p+q_p^*$). All in all, $2v_0-v-r_p^* \in N$.\qedhere
    \end{itemize}
\end{proof}
\end{claim}

The numbers $2v_0-v-r_p^*$, $2v_0-v$ and $2v_0-v+r_p^*$ with $v\in [v_0-p+q_p^*+1,v_0]$ are 
all different. 
This is a consequence of this interval
having length $p-q_p^*-1<r_p^*$ and the distance between any of the three numbers being at least $r_p^*$. Hence, we have found an injection of $[v_0-p+q_p^*+1,v_0-1]\cap S$ (every possible low special integer except $v_0$) into 
$N\cup [q_p^*,p-1]$. Hence we infer that 
\begin{equation}
\label{chain2}
\frac p3<|S|- |N|\leq 1+(p-1-q_p^*+1)\leq 1+r_p^*<\frac p3+1.
\end{equation}
Since the interval $[\frac p3,\frac p3+1]$ contains exactly one integer, we conclude from these
inequalities that $r_p^*=p-q_p^*=\left\lfloor\frac p3\right\rfloor$.

If $p \equiv 2 \,(\mmod{3})$, we can strengthen the inequalities from \lref{boundsp+} and \lref{boundsp-} to
\begin{equation*}
|S|+|P^+|\geq\frac{2}{3}(p+1)
\text{~and~}
|S|+|P^-|\geq\frac{2}{3}(p+1),
\end{equation*}
since $\frac 23(p+1)$ is the 
smallest integer $\ge\frac 23 p$. Now the conclusion of Lemma \ref{S-N} can be sharpened to
 $|S|-|N| \geq \frac{p+4}3$, which cannot be true by the previous paragraph by \eqref{chain2}. 
Therefore the only possible case left is $p \equiv 1 \,(\mmod{3})$, whence 
\begin{equation}
\label{specific}    
p-q_p^*=r_p^* =v_0 =\frac{p-1}3
\end{equation}
by the 
conditions of Case 2.

\begin{claim}\label{fgpairs}
For every $v\in [0,v_0-1]$, we have either $v\in S$ or $f(v)\in S$. Similarly, either $v\in S$ or $g(v)\in S$.
\begin{proof}
By \lref{shiftclass} and since
$f$ is an involution at most one of $v$ 
and $f(v)$ is a special integer.
Note that $f$ 
maps $[0,v_0-1]$ into $[q_p^*,p-1]$. Since we already 
deduced that $v_0=r_p^*=(p-1)/3$, the function $f$ is even a bijection. This immediately shows $|S|\leq r_p^*+1$. Since $|S|>\frac p3$, this inequality must be an equality, so all pairs $(v,f(v))$ must contain a special integer. This means we also need to have $|N|=0$ and hence $S=S_0$.

Now $g$ is also a bijection between these two intervals and $v\in S(=S_0)$ implies $g(v)\notin S$ by \lref{2v0+rp*-v}. Hence, either $v$ or $g(v)$ is special.
\end{proof}
\end{claim}
\par We get $\overline{q_p} = 3$ from $q_p^*=\frac{2p+1}3$. Given any $0 \le v \le p-1$, 
let $k_v \in [0,p-1]$ be the unique integer satisfying 
$([v]p+vq)r+p+q \equiv i+1+k_v\,(\mmod{pq})$ given by \lref{modchi}. By considering 
$k_v$ modulo $p$, we infer that $k_v \neq k_{v'}$ if $v \neq v'$. Therefore, \eqref{p-q_p} is satisfied by exactly three  integers $v$ with $0 \le v \le p-1$ 
(as $\overline{q_p}=3$) and \eqref{q_p} is satisfied by exactly $p-\overline{q_p}=p-3$ integers $v$ with $0 \le v \le p-1$.
We will show that $v_0$ satisfies \eqref{p-q_p} and find three more $v$ for which
this is also the case. These four musketeers then will lead us to victory.
\begin{claim}\label{scupp-}
For every $v'\in [q_p^*,p-1]$, we have $v'\in S\cup P^-$.
\begin{proof}
Assume otherwise. Since $N$ is empty, it follows that $v'\in P^+$, so by \dref{largedef} ($P^+$), $v'\in C_{1,2}$ and $[v']\geq p_q^*$. Observe that $v'+r_p^*-p\leq v_0-1$ is low. Since $f(v'+r_p^*-p)=v'\in P$ is not special, by \clref{fgpairs}, we know $v'+r_p^*-p\in S$.  By definition, this means $v'+r_p^*-p\in C_{1,1}$ and thus $[(v'+r_p^*-p)-r_p^*]\leq p_q^*-1$, which contradicts $[v'-p]=[v']\geq p_q^*$. Hence, such a $v'$ cannot exist.
\end{proof}
\end{claim}

Using \clref{scupp-}, we can apply \lref{s-shift-p} to every $v\in [q_p^*,p-1]$. This implies $h_q(v)\ne h_q(v_0)$ for any $v\in [v_0+1,q_p^*-1]$. 
Hence, none of the integers in $[v_0+1,q_p^*-1]=[\frac{p+2}3,\frac{2p-2}3]$ 
satisfies $h_q(v) = h_q(v_0)$. Since $p\geq 13$ (as the prime $11\not\equiv 1\,(\mmod 3)$ is excluded from consideration), this interval 
contains $\frac{p-1}{3}>3$ integers. Since the number of $v$ 
with $h_q(v)=h_q(v_0)$ is either $3$ or $p-3$, 
it follows that there are exactly three integers $v$ with 
$h_q(v)=h_q(v_0)$. In particular, $v_0$ satisfies \eqref{p-q_p}.
\begin{claim}\label{2lowspecial}
There is a special integer $v_1$ with $0\leq v_1<v_0$.
\end{claim}
\begin{proof}
Assume otherwise, i.e. $v\notin S$ for every $v\in[0,v_0-1]$. By \clref{fgpairs}, this implies $v'\in S$ for every $v'\in[q_p^*,p-1]$.

Since the two values of $h_q(v)$ appear $p-3$, respectively, three times, there must be at least one $\tilde v\neq v_0$ with $h_q(\tilde v)=h_q(v_0)$.  By \lref{chiclass}, such a $\tilde v$ cannot be in the interval $[q_p^*,p-1]$, since all of those integers are special.

In addition, we can apply \lref{s-shift-p} to any $v'\in [q_p^*,p-1]$ (since $v'-r_p^*\in P$ follows from $|N|=0$). Then $h_q(v'-r_p^*)\ne h_q(v_0)$, so $\tilde v\notin[q_p^*-r_p^*,p-1-r_p^*]=[v_0+1,q_p^*-1]$.

We can also apply \lref{s-shift-p.2} to those $v'$ (since $v'+r_p^*-p\in P$ follows from the assumption that no integer in $[0,v_0-1]$ is special, as $v'\in S$
and so $f(v')=v'+r_p^*-p$ is not in $S$ and hence in $P$). Then $h_q(v'+r_p^*-p)\ne h_q(v_0)$, and so $\tilde v\notin [0,v_0-1]$.

Combining the above, we see that such a $\tilde v$ cannot exist. Therefore, a low special integer $v_1\neq v_0$ must exist. As $v_0$ is the largest special low integer, we have $v_1<v_0$.
\end{proof}
Let $v_1<v_0$ be a low special integer, which exists by \clref{2lowspecial}.
We get $f(v_1)=v_1-r_p^*+p \notin S$ (\lref{shiftclass}). By \lref{fgpairs}, $v_1-r_p^*+p\notin S$ implies $g(v_1-r_p^*+p)\in S$, so $g(v_1-r_p^*+p)=2v_0+2r_p^*-v_1-p=r_p^*-1-v_1\in S$.
\par Since $v_0$ satisfies \eqref{p-q_p} and each of 
$v_0,v_1$ and $r_p^*-1-v_1$ is both low
and special, they all satisfy \eqref{p-q_p} by \lref{sameclass}.
Using \eqref{specific} we see that $q_p^*\le v_1-r_p^*+p\le p-1$. As  
$v_1-r_p^*+p\notin S$ and $N$ is empty, we conclude that 
$v_1-r_p^*+p\in P.$ Recalling that $v_1\in S$, it now follows by
\lref{+shift} that also $v_1-r_p^*+p$ satisfies \eqref{p-q_p}.
We have $\max\{v_1,r_p^*-1-v_1\}<v_0<v_1-r_p^*+p$, 
and thus if two of these four integers are equal, then necessarily  $v_1=r_p^*-1-v_1$ and so $v_1=\frac{p-4}6$. However, $v_1$ is not an integer since $p-4$ is odd. 
Thus we have identified four integers in $I_p$
satisfying \eqref{p-q_p}, giving rise to a contradiction.\\

\noindent \textbf{Case 3:} $q_p^*-r_p^* \leq v_0 \leq q_p^*-1$. \\
Let $v$ be any high special integer.
If $v \notin S_0$, then $f(v) \in N$, so $v$ does not contribute to the difference $|S|-|N|$ and hence $|S|-|N|\leq |S_0|$. 
It remains to 
deal with those $v$ that are in $S_0$, which we will 
do by considering two $v$-ranges separately.\\ 
\indent If $q_p^* \leq v \leq v_0+r_p^*$, then $v_0-r_p^*+1 \leq q_p^*-r_p^*\leq v-r_p^* \leq v_0$. Moreover, $v\in S_0$ (by
assumption), so $f(v)\in P$ and by \lref{s-shift-p}, we even have $f(v)\in P^+$. Thus, we can restrict $f$ to a smaller domain, giving rise to a map $\tilde f:S_0\cap [q_p^*,v_0+r_p^*]\rightarrow P^+\cap [v_0-r_p^*+1,v_0]$.\\
\indent If $v_0+r_p^*+1 \leq v \leq p-1$, 
then $g(v)=2v_0+r_p^*-v \in [v_0-r_p^*+1,v_0]$ (since $2v_0+r_p^*-p+1\geq v_0+q_p^*-r_p^*+r_p^*-p+1\geq v_0-r_p^*+1$) and hence $2v_0+r_p^*-v \notin S\cup P^+$ by \lref{2v0+rp*-v}. Thus, we have $\tilde g:S_0\cap [v_0+r_p^*+1,p-1]\rightarrow [v_0-r_p^*+1,v_0]$ with $\operatorname{Im}(\tilde g)$ and $S\cup P^+$ being disjoint.
\par Observe that the domains of $\tilde f$ 
and $\tilde g$ cover $[q_p^*,p-1]\cap S_0$ and thus 
they cover every high integer in $S_0$. Also, the ranges 
of $\tilde f$ 
and $\tilde g$
are distinct, since $\operatorname{Im}(\tilde f)\subseteq P^+$, whereas $\tilde g(v)\notin P^+$.
All low integers in $S_0$ are contained in $[v_0-p+q_p^*+1,v_0]\subseteq [v_0-r_p^*+1,v_0]$
by \cref{bound}.
Hence, the map
\[h:S_0\to [v_0-r^*_p+1,v_0],\quad
v\mapsto\begin{cases}
v&\text{for $v\in [0,q_p^*-1]$};\\
\tilde f(v)&\text{for $v\in [q_p^*,v_0+r_p^*]$};\\
\tilde g(v)&\text{for $v\in [v_0+r_p^*+1,p-1]$},
\end{cases}\]
is injective,
and it follows that $|S|-|N|\leq |S_0|\leq r_p^*<\frac p3$.\\

Combining the three cases above we obtain 
$|S|-|N|<\frac{p}3$ for every possible $v_0$. 
By Lemma \ref{S-N} it then follows that 
$A(pqr)\le\frac 23 p$ and thus the corrected Sister Beiter cyclotomic coefficient conjecture is proven.
\subsection{On the restriction $p\ge 11$}
Note that we really used $p \ge 11$ in Case 2,
as for our 
argument the number of integers in $\left[\frac {p+2}3,\frac {2p-2}3\right]$ 
has to exceed $3$. This is not true for $p=7$ (and neither for $p=11$, but $11\not\equiv 1\,(\mmod{3})$). Nevertheless, it is possible to finish the proof with very similar arguments for $p=7$, {cf.\,Zhao and Zhang \cite{ZZ1}.
For $p\in \{3,5\}$ it immediately follows from \tref{Bzd} that
$M(p)\le 2p/3$.
\subsection{Establishing the weaker bound $M(p)\le (2p+1)/3$}
If one is satisfied with proving that $M(p)\le \frac{2p+1}{3}$, a 
shorter proof is possible. In Case 2, after establishing that 
$q_p^*=\frac{2p+1}{3}$, one then merely concludes using \tref{tabletheorem}
that $M(p;q)\le f(\frac{2p+1}3)=f(\frac{p-1}3)=\frac{2p+1}3$. Further,
there is no need to formulate and prove \lref{s-shift-p.2} and
\lref{+shift}.
\section{Improvement of some bounds of Bzd\k{e}ga}
\label{bartekbounds}
The ternary coefficient bounds of Bzd\k{e}ga given in 
\tref{Bzd} are quite
useful. Combination of our main result with his, then leads to the following
improvement.
\begin{theorem}
\label{Bzdmodern} 
Let $p<q<r$ be primes. Let
\begin{equation*}
\alpha = \min\{q_p^*,~p-q_p^*,~r_p^*,~p-r_p^*\}
\end{equation*}
and let $0<\beta\leq p-1$ the unique integer with $\alpha\beta qr \equiv 
1\,(\mmod{p})$. 
We have
\begin{equation*}
a_{pqr}(i) \leq \min\{2\alpha + \beta,~p-\beta,2p/3\}\text{~and~}-a_{pqr}(i) \leq \min\{p+2\alpha - \beta,~\beta,2p/3\}.
\end{equation*}
\end{theorem}
Recall the definition \eqref{auxiliary} of $m$ and $w$.
Using Bzd\k{e}ga's bounds the following upper bound for $M(p;q)$ can be derived.
\begin{theorem}[Gallot et al.\,\cite{GMW}, 2011]
\label{BB}
Let $p<q$ be primes with $q\equiv \beta\,(\mmod{p}).$
Then $M(p;q)\le M_{\beta}(p)\le w(\beta^*)$.
\end{theorem}
This in combination with our main result leads to the following
sharpening.
\begin{theorem}
\label{BBimproved}
Let $p<q$ be primes with $q\equiv \beta\,(\mmod{p}).$
Then $M(p;q)\le M_{\beta}(p)\le m(\beta^*)$.
\end{theorem}
Since $\#\{1\le j\le p-1:m(j)<w(j)\}$ asymptotically grows as
$p/6$, \tref{BBimproved} is a true improvement of \tref{BB}.
\section{The proof of Theorem \ref{tabletheorem}}
Let $v$ be a non-zero entry in the $p$-row and $\beta$-column with 
$\beta\le \frac{p-1}{2}$ in Table {\rm 2}. 
By Theorem \ref{BBimproved}  we have $M_{\beta}(p)\le m(\beta^*)$. It remains to
establish the lower bound $M_{\beta}(p)\ge v$.
Since $M_{\beta}(p)\ge M(p;q)$ for 
$q\equiv \beta\,(\mmod{p})$, it is enough for the non-boldface cases to find
one example with $M(p;q)=v$ with 
$q\equiv\beta\,(\mmod{p})$. This also holds for the boldface cases where 
$v=M(p)$. 
We will not give explicit examples, but note that the reader can work 
some out from Table 1.
However, more satisfactory than finding one example, is to find a construction which yields  $M(p;q)=v$ with 
$q\equiv \beta\,(\mmod{p})$ for all primes $q$ large enough. 
Usually the results we quote below rest on constructions given in 
the indicated references, e.g., in \cite{GM} one finds a construction for $M_5(13)\ge 8$. In the preprint of Gallot et al.\,\cite{GMWpre} more details of the construction are in general supplied than in the published version \cite{GMW}.
\par By  \cite[Thm.\,27]{GMW} we have $M_1(p)=M_2(p)=\frac{p+1}{2}$, and so we may assume that $\beta\ge 3$.
\par We finish the proof by discussing the six relevant primes individually.\\ 
$p=3$. If $q > 3$ is a prime, then $M(3; q) = 2$ \cite[Thm.\,17]{GMW}.\\
$p=5$. If $q > 3$ is a prime, then $M(5; q) = 3$ \cite[Thm.\,28]{GMW}.\\
$p=7$. If $q >13$ is a prime,
then $M(7; q) = 4$ \cite[Thm.\,32]{GMW}.\\
$p=11$. The result for $p=11$ follows from \cite[Thm.\,36]{GMW} (which assumes
the Corrected Sister Beiter conjecture to be true).\\
$p=13$. Theorem 4 of 
\cite{GM} (together with the corresponding Table 1 in that paper) shows that $M_{5}(13)=8.$\\
$p=19$. Theorem 4 of 
\cite{GM} (together with the corresponding Table 1 in that paper) shows that $M_{8}(19)\ge 11.$ Lemma 38 of \cite{GMW} together with $M(19)\le 12$ shows that
$M_{4}(19)=12$ (and hence $M(19)=12$).

\subsection*{Acknowledgment}
Work on this paper was started during an internship of the first and third author and finished during an internship of the fourth and fifth author at the  Max Planck Institute for Mathematics (MPIM) in Bonn with the second author in September 2015, respectively September 2022. 
The first and third author made a good start with clarifying the argument
of Zhao and Zhang, the fourth and fifth author finished this process. The second author polished the write-up, wrote
the introduction, and added the material on applications.
\par The second author
thanks the administration and Christian Blohmann for their support in organizing the internships.
The remaining authors like to thank MPIM for hosting them and the hospitality. Thanks are also due for Jia Zhao for supporting this new version of his purported proof with Xianke Zhang. 
Yves Gallot and Bin Zhang kindly helped with the computations in Table 1.
Table 2 is based on a lot of computational work by Robert Wilms and, especially,
Yves Gallot. Alessandro Languasco helped with some
LaTeX issues we had. Gennady Bachman pointed out the relevance of \cite{BM}.

\end{document}